\documentclass[12pt]{article}
\usepackage{geometry}                
\usepackage{graphicx,color}
\usepackage{amssymb,amsmath,amsthm,mathrsfs}
\usepackage[all,cmtip]{xy}
\usepackage{epstopdf, comment}

\DeclareMathOperator{\diam}{diam }
\DeclareMathOperator{\dist}{dist }

\DeclareMathOperator{\loc}{loc }
\DeclareMathOperator{\SO}{SO }

\DeclareMathOperator{\aut}{Aut }

\DeclareMathOperator{\sph}{\hat{\mathbb{C}}}
\DeclareMathOperator{\haus}{Haus }

\newtheorem{theorem}{Theorem}[section]
\newtheorem{corollary}[theorem]{Corollary}
\newtheorem{lemma}[theorem]{Lemma}

\theoremstyle{remark}
\newtheorem*{remark}{Remark}

\numberwithin{equation}{section}

\title{On meromorphic functions whose \\ image has finite spherical area}
\author{Oleg Ivrii}

\begin{document}

\maketitle

\begin{abstract}
In this paper, we study meromorphic functions on a domain $\Omega \subset \mathbb{C}$ whose image has finite spherical area, counted with multiplicity.
The paper is composed of two parts. In the first part, we show that the limit of a sequence of meromorphic functions is naturally defined on $\Omega$ union a tree of spheres. In the second part, we show that a set $E \subset \Omega$ is removable if and only if it is negligible for extremal distance.
\end{abstract}

\section{Introduction}

Let $\Omega \subset \mathbb{C}$ be a domain in the complex plane. A holomorphic function on $\Omega$ belongs to the classical Dirichlet space $\mathcal D = \mathcal D(\Omega)$ if the Euclidean area of its image counted with multiplicity is finite:
\begin{equation}
\label{eq:def-dirichlet}
\mathscr A_{\mathbb{C}}(F, \Omega) = \int_\Omega |F'(z)|^2 \, |dz|^2 < \infty.
\end{equation}
In this paper, we present two mostly independent vignettes on the ``spherical Dirichlet space'' $\mathcal F = \mathcal F(\Omega)$ which consists of meromorphic functions on $\Omega$ whose images have finite spherical area, i.e.~
\begin{equation}
\label{eq:spherical-area-bound}
\mathscr A_{\sph}(F, \Omega) = \int_\Omega \biggl ( \frac{2|F'(z)|}{1 + |F(z)|^2} \biggr )^2 |dz|^2 < \infty.
\end{equation}

\subsection{Bubbling of meromorphic functions}

Suppose that $\{F_n\}$ is a sequence of meromorphic functions on a domain $\Omega \subset \mathbb{C}$ for which 
\begin{equation}
\label{eq:spherical-area-bound}
\mathscr A_{\sph}(F_n, \Omega) = \int_\Omega \biggl ( \frac{2|F'_n(z)|}{1 + |F_n(z)|^2} \biggr )^2 |dz|^2 \le C.
\end{equation}
If $C < 4\pi$, then $\{F_n\}$ is a normal family since the image of each $F_n$ misses a positive area subset of the sphere. In general, $\{F_n\}$ is 
only {\em quasinormal}  in the sense that a subsequence  converges locally uniformly in $\Omega \setminus S$, where $S$ is a finite set. 

We assume that no point $p \in S$ is redundant, i.e.~that $\{F_n\}$ is not normal in any neighbourhood of $p$. This implies that for any $p \in S$ and $r > 0$, 
$$\mathscr A_{\sph} (F_n, B(p,r)) \ge 4\pi,$$
for all $n$ sufficiently large. In particular, the cardinality of $S$ is at most $\lfloor C/4\pi \rfloor$. 
We  pass to a further subsequence so that the measures
$$
\mu_n = F_n^*dA_{\sph} = \biggl ( \frac{2|F'_n(z)|}{1 + |F_n(z)|^2} \biggr )^2 |dz|^2
$$ 
converge weakly to a measure $\mu$.
The limiting measure $\mu$ may have point masses at the points of $S$. In \cite{GKR}, Grahl, Kraus and Roth  observed that $\mu(\{p\}) \ge 4\pi$ for any $p \in S$, and 
suggested that mass is quantized: $$\mu(\{p\}) = 4\pi D, \qquad \text{for some integer }D \ge 1.$$
It is intuitively clear that a part of the structure is lost when one takes the naive pointwise limit $F: \Omega \to \sph$.
After reading their manuscript, the author showed the following theorem which describes the ``full limit'' of the meromorphic functions $F_n: \Omega \to \sph$\,:

\begin{theorem}
\label{bubbling-thm}
After passing to a subsequence, the full limit of the $F_n$ is naturally a meromorphic function defined on a multi-nodal surface $X$ obtained by gluing trees of spheres to $\Omega$ at points of $S$. The full limit consists of:
\begin{itemize}
\item A meromorphic function $F: \Omega \to \sph$,  obtained as the pointwise limit of the functions $F_n: \Omega \to \mathbb{C}$ on $\Omega \setminus S$.
\item Each sphere $\Sigma_k \subset X$, $k = 1, 2, \dots, N$, comes equipped with a rational function $\mathcal R_k: \Sigma_k \to \hat{\mathbb{C}}$ of degree $D_k \ge 1$. These rational functions are obtained as appropriate rescaling limits of the $F_n$ and are determined uniquely up to pre-composition with affine maps $z \to az + b$ in $\aut \mathbb{C}$.
\end{itemize}
The limiting mass is $\mu(\{p\}) = 4\pi D$, where $D = \sum D_k$ is the sum of the degrees of the rational maps associated to spheres contained in trees attached at $p$.
\end{theorem}

\begin{remark}
A {\em multi-nodal surface} $X$ is a topological space obtained by gluing a finite or countable collection of surfaces $\{ \mathcal S_k \}$ at a discrete set of points, so that each point in $X$ is contained in at most finitely surfaces $\mathcal S_k$. A point in $X$ is called a {\em regular point} if it is contained in exactly one surface  $\mathcal S_k$ and a {\em multi-node} if it is contained in more than one surface. (In the literature, a {\em node} refers to a multi-node that is contained in exactly two surfaces. A {\em nodal surface} is a multi-nodal surface, where each multi-node is at worst a node.)
\end{remark}


Actually, Theorem \ref{bubbling-thm} was already known in much greater generality: it was proved for pseudo-holomorphic mappings into almost K\"ahler manifolds by Parker and Wolfson \cite{PW} and for quasiregular mappings by Pankka and Souto \cite{PS}. The locally univalent case has also been obtained by Li and Shafrir \cite{LS} using PDE techniques. Nevertheless, we hope that our elementary argument will be accessible to a wider audience as it involves only elementary notions from complex analysis and topology.

One definition from topology that will feature prominently in our argument is that of an {\em orientation-preserving branched cover} $\hat{F}: \sph \to \sph$\,: a continuous self-mapping of the sphere, which is a local homeomorphism outside of a finite set of points, where it is topologically equivalent to $z \to z^n$ for some $n \ge 2$.
 If the topological degree of $\hat{F}$ is $d$, then $\hat{F}$ has $2d-2$ critical points counted with multiplicity (as is the case for a rational function of degree $d$) and every point $w \in \sph$ that is not a critical value  is covered exactly $d$ times. In particular, the spherical area of the image (counted with multiplicity) $\mathscr A_{\sph}(F, \sph) = 4\pi d$.

\subsection{Removable sets}
\label{chap:removable}

A compact set $E \subset \Omega$ is called {\em removable} for Dirichlet functions if any holomorphic function in $\mathcal D(\Omega \setminus E)$ extends to a holomorphic function in $\mathcal D(\Omega)$. Similarly, we say that a set $E \subset \Omega$ is  removable for spherical Dirichlet functions if any meromorphic function in $\mathcal F(\Omega \setminus E)$ extends to a meromorphic function in $\mathcal F(\Omega)$.

In a classical work  \cite{AB}, Ahlfors and Beurling showed that a set is removable for Dirichlet functions if and only if it is a NED (negligible for extremal distance) set. 
A set $E$ is NED if for any rectangle $\mathscr R$, the modulus of curves connecting a pair of opposite sides that avoid $E$ computes the modulus of $\mathscr R$.
NED sets have a number of other characterizations, for example, a set $E$ is NED if any conformal embedding $\mathbb{C} \setminus E \to \mathbb{C}$ is linear.

In the second half of the paper,  we show that the NED condition also characterizes removable sets for spherical Dirichlet functions:

\begin{theorem}
\label{removability-thm}
A compact set $E$ is removable for spherical Dirichlet functions if and only if it is NED.
\end{theorem}


One direction is easy:

\begin{lemma}
Suppose a compact set $E \subset \Omega$ is removable for spherical Dirichlet functions. Then, it is also removable for usual Dirichlet functions and therefore NED.
\end{lemma}

\begin{proof}
We first notice that any $\mathcal F$-removable set $E$ has 2-dimensional Lebesgue measure 0. Otherwise, one can find a non-trivial quasiconformal homeomorphism $w^\mu: \mathbb{C} \to \mathbb{C}$ whose dilatation is supported on $E$. As $w^\mu$ is injective, $\mathscr A_{\hat{\mathbb{C}}}(w^\mu, \Omega \setminus E) \le 4\pi$ and $w^\mu \in \mathcal F(\Omega \setminus E)$. Since $w^\mu$ is conformal on $\Omega \setminus E$ but not conformal on $\Omega$, it cannot extend to a meromorphic function on $\Omega$.

Suppose $F$ is a holomorphic function in $\mathcal D(\Omega \setminus E)$. If $E$ is removable for the class $\mathcal F$, then $F$ extends to a meromorphic function on $\Omega$ which satisfies
$$
 \int_\Omega |F'(z)|^2 \, |dz|^2 =  \int_{\Omega \setminus E} |F'(z)|^2 \, |dz|^2 < \infty.
$$
As the last condition is incompatible with $F$ having poles, $F$ is holomorphic on $\Omega$. Thus, $F \in \mathcal D(\Omega)$ as desired.
\end{proof}


The original argument of Ahlfors and Beurling which showed that NED sets are Dirichlet removable used a special identity involving conformal maps to slit domains, while a newer proof by Hedberg \cite{hedberg} utilized the connection between  the Dirichlet energy and the condenser capacity. Since these miraculous connections are unavailable in the spherical setting, it is unlikely that one can extend these approaches to spherical Dirichlet functions.
 Our proof is a slight improvement of the beautiful isoperimetric argument of Parker and Wolfson \cite{PW} who observed that points are removable for spherical Dirichlet functions.
 
 In the case when $F$ is defined on the complex plane, one can show that the singularity at infinity is removable using Nevanlinna theory:
 
\begin{lemma}
\label{removable-singularities}
Suppose $F: \mathbb{C} \to \hat{\mathbb{C}}$ is a meromorphic function whose image has finite spherical area, counted with multiplicity:
$$
\mathscr A_{\sph}(F, \mathbb{C}) = \int_{\mathbb{C}} \biggl ( \frac{2|F'(z)|}{1 + |F(z)|^2} \biggr )^2 |dz|^2 < \infty.
$$
Then $F$ is a rational function. If $\deg F = d$ then $\mathscr A_{\hat{\mathbb{C}}}(F, \mathbb{C}) = 4\pi d$.
\end{lemma}

\begin{proof}From the Ahlfors-Shimizu interpretation of Nevanlinna theory \cite[Section 1.5]{hayman}, it is easy to see that the function $F$ can attain any value $a \in \hat{\mathbb{C}}$ only a finite number of times. One merely needs to use the fact that the Nevanlinna counting function is bounded above by the Ahlfors-Shimizu characteristic:
$$
\sum_{\substack{F(z) = a \\ |z| < r}} \log \frac{r}{|z|} \le \frac{1}{4\pi} \int_0^r \mathscr A_{\sph}(F, B(0,t)) \cdot \frac{dt}{t} + O(1), \qquad F(0) \ne a.
$$
By Picard's theorem, the singularity at infinity cannot be essential, and therefore, is at worst a pole. Hence, $F$ is a rational function.
\end{proof}
We also mention a third approach to removability due to Chen and Li \cite{CL} who used the method of moving planes to prove Lemma \ref{removable-singularities} for locally univalent functions. While their result is unable to handle holomorphic functions with critical points, it is applicable to other non-linear PDEs and works in higher dimensions.

\subsection{Rescaling limits}

Suppose $F: \Omega \to \hat{\mathbb{C}}$ is a meromorphic function. A {\em rescaling} of $F$ is a map of the form $F \circ m$ where $m$ is a M\"obius transformation. 

\begin{corollary}
\label{rescaling-corollary}
Suppose that $\{F_n\}$ is a sequence of meromorphic functions on a domain $\Omega \subset \hat{\mathbb{C}}$ satisfying (\ref{eq:spherical-area-bound}) and $\tilde F_{n} = F_n \circ m_n$ is a sequence of rescalings of $F_n: \Omega \to \hat{\mathbb{C}}$ defined on balls $B(0, R_n)$ with $R_n \to \infty$. After passing to a subsequence, the maps  $\tilde F_{n}$ converge locally uniformly to a rational function $\mathcal R$ on $\mathbb{C}$ minus a finite set of points.
\end{corollary}

\begin{proof}
Since the spherical areas
$$
\mathcal A_{\sph}\bigl (\tilde F_n, B(0,R)\bigr ) \, = \, \mathcal A_{\sph}\Bigl (F_n, m_n(B(0,R_n) \Bigr ) \, \le \, C,
$$
are uniformly bounded, the sequence $\tilde F_n$ is quasi-normal on any ball $B(0, R)$. After passing to a subsequence, the $\tilde F_n$ converge to a meromorphic function
$\tilde F: \mathbb{C} \to \hat{\mathbb{C}}$ satisfying
$$
\mathcal A_{\sph}\bigl (\tilde F, \mathbb{C} \bigr ) \, = \, 
\lim_{R \to \infty} \mathcal A_{\sph}\bigl (\tilde F, B(0,R) \bigr ) \, \le \,
 \lim_{R \to \infty} \Bigl \{ \liminf_{n \to \infty} \mathcal A_{\sph}\bigl (\tilde F_n, B(0,R) \bigr )  \Bigr \} \, \le \, C,
$$
 locally uniformly outside a finite set of cardinality at most $\lfloor C/4\pi \rfloor$.
From Lemma \ref{removable-singularities}, it follows that $\tilde F = \mathcal R$ is a rational function.
\end{proof}

We say that a sequence of rescalings $\tilde F_n$ is {\em trivial} if the limit $\mathcal R$ is a constant function. The rational functions $\mathcal R_k$, $k=1,2,\dots, N$ in Theorem \ref{bubbling-thm} will be obtained as rescalings of the sequence $\{F_n\}$ near points of $S$.  One difficulty in Theorem \ref{bubbling-thm} is to figure out where to rescale. This is a somewhat delicate matter: for instance, in the proof of Zalcman's lemma \cite{marshall}, one rescales at local maxima of certain functions associated with the $F_n$. This approach is guaranteed to provide at least one rescaling limit, but in general, is unable to give the full set of non-trivial rescaling limits. The approach of Li and Shafrir \cite{LS} uses a similar idea and therefore suffers from the same drawback.

\section{Mass quantization}

Since the issue of bubbling is a local matter, by shrinking the domain $\Omega$, we may assume that the set $S = \{ p \}$ consists of a single point. For simplicity of exposition, we assume that the sequence $F_n$ satisfies the following three conditions:
\begin{enumerate}
\item The limit function $F$ is not constant,
\item  $F$ is holomorphic near $p \in S$,
\item $p$ is not a critical point of $F$, i.e.~$F'(p) \ne 0$. 
\end{enumerate}
We call such sequences of meromorphic functions {\em elementary}.
In Section \ref{sec:modifications},  we will explain how to handle the case of general sequences.

\begin{lemma}[Mass quantization]
Suppose $F_n: \Omega \to \sph$ is an elementary sequence of meromorphic functions with $\mathscr A_{\sph}(F_n, \Omega) \le C$.  After passing to a subsequence, we may assume that the measures $\mu_n = F_n^*dA_{\sph}$ converge weakly to a measure $\mu$ on $\Omega$. Then,  $\mu(\{p\}) = 4\pi D$ for some integer $D \ge 1$.
\end{lemma}

\begin{proof}
Choose $r > 0$ sufficiently small so that $F$ is univalent on $B(p,r)$ and the image $F(\overline{B(p,r)})$ is homeomorphic to a closed disk.
In particular, $B(p, r) \subset \Omega$ contains no critical points of $F$. By discarding countably many  radii,
we may assume that each map $F_n$ has no critical points on $\partial B(p,r)$, although it may have plenty of critical points inside $B(p,r)$. 

We orient the boundary $\partial B(p,r)$ counter-clockwise. By assumption $(2)$, the image   of $\partial B(p,r)$ under $F$ is a Jordan curve which winds counter-clockwise around $F(p)$. Since $F_n \to F$ uniformly on a neighbourhood of $\partial B(p,r)$, for $n$ sufficiently large, $F_n(\partial B(p,r))$ will also be a Jordan curve which winds counter-clockwise around $F(p)$.

We  extend $F|_{B(p,r)}$ to an orientation-preserving diffeomorphism $\hat F: \sph \to \sph$ by selecting an orientation-preserving diffeomorphism from $\sph \setminus B(p,r)$ onto  $\sph \setminus F(B(p,r))$ that agrees with $F$ on $\partial B(p,r)$.

In a similar way, we  extend each $F_n|_{B(p,r)}$ to an orientation-preserving branched cover $\hat F_n: \sph \to \sph$ by selecting an orientation-preserving diffeomorphism from
$\sph \setminus B(p,r)$  onto  $\sph \setminus F_n(B(p,r))$ that agrees with $F_n$ on $\partial B(p,r)$. Since the $F_n$ converge uniformly to $F$ on $\partial B(p,r)$, one can choose the extensions $\hat F_n|_{\sph \setminus B(p,r)}$  which tend to $\hat F|_{\sph \setminus B(p,r)}$ uniformly in the spherical metric.  

Let $d_n$ be the topological degree of  $\hat F_n$. Since $\hat F_n$ is a $d_n:1$ mapping of  the Riemann sphere to itself,
 $$
 \mathscr A_{\sph} \bigl (\hat F_n, {B(p,r)} \bigr )  +  \mathscr A_{\sph}\bigl (\hat F_n, {\sph \setminus B(p,r)} \bigr ) = 4\pi d_n.
 $$ 
   As $\mathscr A_{\sph}(\hat F_n, {\sph \setminus B(p,r)}) < 4\pi$, the degrees $d_n$ are uniformly bounded above.   We  pass to a subsequence for which $\{d_n\}$ is constant. 
   
Since $d = 1$ is the topological degree of $\hat F$, 
 $$
 \mathscr A_{\sph} \bigl (\hat F, {B(p,r)} \bigr  )  +  \mathscr A_{\sph}\bigl (\hat F, {\sph \setminus B(p,r)} \bigr ) = 4\pi d.
 $$ 
  It is clear that the degree of $\hat F_n$ can only drop in the limit, i.e.~$d \le d_n$, and bubbling occurs when there is a strictly inequality: $d < d_n$. 
  
  As  
$$
\mathscr A_{\sph}\bigl (\hat F_n, {\sph \setminus B(p,r)} \bigr ) \to \mathscr A_{\sph}\bigl (\hat F, {\sph \setminus B(p,r)}\bigr ),
$$
the spherical area of $F_n(B(p,r))$ can only drop by a multiple of $4\pi$, i.e.~
$$
\lim_{n \to \infty} \mathscr A_{\sph}\bigl (\hat F_n,  B(p,r) \bigr) = \mathscr A_{\sph}\bigl (\hat F,  B(p,r) \bigr ) + 4\pi D,
$$
where $D = d_n - d$.
The $4\pi D$ drop in area must also be accompanied with a drop in $2D$ critical points. 
Since the convergence $F_n \to F$ is uniform away from $p$, by Hurwitz theorem, these critical points must tend to $p$ as $n \to \infty$.
\end{proof}

\section{Tree structure}
\label{sec:tree}

In this section, we prove Theorem \ref{bubbling-thm} for elementary sequences of meromorphic functions.
Continuing the discussion from the previous section, we pass to a further subsequence so that the critical values of $\hat F_n$ converge in $\hat{\mathbb{C}}$ and denote the limiting critical value set by
$\mathscr V$. 
Pick a round ball $\mathscr B = B(w, \rho)$ so that 
$$B(w, 2\rho) \subset \mathbb{C} \setminus \bigl (\mathscr V \cup F(B(p,r)) \bigr ).$$
The pre-image $\hat F_n^{-1}(\mathscr B)$ consists of $d_n$ connected components.
Since exactly  one of these connected components is  outside $B(p,r)$, the remaining   $d_n-1$ components are compactly contained in $B(p,r)$, which we label $\mathscr B_1^{(n)}, \mathscr B_2^{(n)}, \dots, \mathscr B_{d_n-1}^{(n)}$. 

By construction, $F_n$ maps each topological disk $\mathscr B_i^{(n)}$ conformally onto
$\mathscr B$ and each Jordan curve $\gamma_i^{(n)} := \partial \mathscr B_i^{(n)}$ homeomorphically onto $\partial \mathscr B$.
Let $z_i^{(n)}$ be the pre-image of $w$ contained in $\mathscr B_i^{(n)}$ and $r_i^{(n)} := \diam \mathscr B_i^{(n)}$.
By Koebe's distortion theorem, the domains $\mathscr B_i^{(n)}$ are uniformly round: there exists constants $C_1, C_2 > 0$, independent of $n$, such that
  $$
B(z_i^{(n)}, C_1 \cdot r_i^{(n)}) \, \subset \, \mathscr B_i^{(n)} \, \subset \, B(z_i^{(n)}, C_2  \cdot r_i^{(n)}).
 $$
 
\paragraph{Clusters.} Passing to a subsequence if necessary, the curves $\gamma_i^{(n)} := \partial \mathscr B_i^{(n)}$ become organized into {\em clusters},
with $\gamma_i^{(n)}$ and $\gamma_j^{(n)}$ in the same cluster if
$$
\diam \gamma_i^{(n)} \, \asymp \, 
\dist(\gamma_i^{(n)}, \gamma_j^{(n)}) \, \asymp \, \diam \gamma_j^{(n)}, \qquad \text{as }n \to \infty.
$$
In other words, we ask that the ratios 
$$
\frac{ \dist(\gamma_i^{(n)}, \gamma_j^{(n)})}{\diam \gamma_j^{(n)}} \qquad
\text{and} \qquad
\frac{\diam \gamma_i^{(n)}}{\diam \gamma_j^{(n)}}
$$
are bounded from 0 and $\infty$.

\paragraph{Spheres.}  We let $N$ denote the number of clusters. By Corollary \ref{rescaling-corollary}, each cluster $\mathscr C_j$ has a non-constant rescaling limit:
$$
F_n(z_j^{(n)} + r_j^{(n)} z) \to \mathcal R_j(z),
$$
where the convergence is uniform on $\mathbb{C} \setminus S_j$, where $S_j$ is a finite set. Since the degree of $\mathcal R_j$ is just the size of the cluster $\mathscr C_j$, the degrees of the rational functions $\mathcal R_j$ add up to  $d_n-1 = D$.

From the construction, it is clear that the rational functions $\mathcal R_j$ are uniquely determined up to pre-composition with affine maps in $\aut \mathbb{C}$.

\begin{lemma}
\label{sum-degrees}
The degrees of the rescaling limits account for the drop in the degree:
$$
\deg \mathcal R_1 + \deg \mathcal R_2 + \dots + \deg \mathcal R_N = D.
$$
\end{lemma}

\paragraph{Tree structure.}  In turn, the clusters are naturally organized into a tree of spheres structure. We say that the sphere $(\Sigma_i, \mathcal R_i)$ is a {\em descendant} of $(\Sigma_j, \mathcal R_j)$ if for any $\gamma_i \in \mathscr C_i$ and $\gamma_j \in \mathscr C_j$,
$$
\frac{\diam \gamma_i^{(n)}}{\diam \gamma_j^{(n)}} \to 0, \qquad 
 \dist(\gamma_i^{(n)}, \gamma_j^{(n)}) \lesssim  \diam \gamma_j^{(n)}, \qquad \text{as }n \to \infty,
$$
and write $(\Sigma_i, \mathcal R_i) \prec (\Sigma_j, \mathcal R_j)$. We view every sphere $(\Sigma_i, \mathcal R_i)$ as a descendant of $(\Omega, F)$.
 The relation $\prec$ is clearly transitive. We say that $(\Sigma_i, \mathcal R_i)$ is a {\em direct descendant} of $(\Sigma_j, \mathcal R_j)$ if there is no index $k$ for which $(\Sigma_i, \mathcal R_i) \prec (\Sigma_k, \mathcal R_k) \prec (\Sigma_j, \mathcal R_j)$.
 
\paragraph{Assembly.}  If $(\Sigma_i, \mathcal R_i)$ is a direct descendant of $(\Sigma_j, \mathcal R_j)$, then we attach $(\Sigma_i, \mathcal R_i)$  to  $(\Sigma_j, \mathcal R_j)$ by gluing the point at infinity in $\Sigma_i$ to 
$$
p_{i,j} \,=\, \lim_{n \to \infty} \frac{z_i^{(n)} - z_j^{(n)}}{r_j^{(n)}} \, \in \Sigma_j.
$$
Similarly, if $(\Sigma_i, \mathcal R_i)$ is a direct descendant of $(\Omega, F)$, then we attach $(\Sigma_i, \mathcal R_i)$ to $(\Omega, F)$ by gluing the point at infinity in $\Sigma_i$ to $p \in \Omega$.

As $\mathcal R_i(\infty) = \mathcal R_j(p_{i,j})$,   the full limit $(F, \mathcal R_1, \dots, \mathcal R_N)$ is continuous on the multi-nodal surface $X$ obtained by gluing $\Omega$ and the spheres $\Sigma_1, \Sigma_2, \dots, \Sigma_N$.
 
\section{General sequences}
\label{sec:modifications}

We now explain how to handle general sequences of meromorphic functions $F_n$.

\medskip

\noindent Q. What happens if $F$ has a pole at $p$ instead of being holomorphic?

\medskip

\noindent A. We can instead work with $1/F_n$. Under this transformation, the spherical derivative remains unchanged:

$$
\frac{2|(1/F_n)'|}{1 + |1/F_n|^2} = \frac{2|F'_n/F_n^2|}{1 + |1/F_n|^2} = \frac{2|F'_n|}{1 + |F_n|^2}.
$$

\medskip

\noindent Q. What happens if $F$ has a critical point at $p$ of order $a \ge 1$?

\medskip

\noindent A. The proof proceeds as before, but this time, we cannot extend $F_n|_{B(p,r)}$ and $F|_{B(p,r)}$ to be one-to-one on $\sph \setminus B(p,r)$.
When defining the extensions  $\hat F_n|_{B(p,r)}$ and  $\hat F|_{B(p,r)}$, we add a critical point of order $a$ at infinity. The extensions are faciliated by the following lemma:

\begin{lemma}
{\em (i)} Suppose $F: \partial B(p,r) \to \mathbb{C}$ is a $C^1$ immersed curve in the plane such that $F(p) \notin F(\partial B(p,r))$ and 
$$
\frac{d}{d\theta} \, \arg \bigl ( F(p + re^{i\theta}) - \arg F(p) \bigr ) > 0.
$$
If  $F(\partial B(p,r))$ winds $k = a + 1$ times counter-clockwise around $F(p)$, then
 $F$ extends to a $C^1$ orientation-preserving branched cover $\hat F: \hat{\mathbb{C}} \setminus \overline{B(p,r)} \to
 \hat{\mathbb{C}}$ which is $k : 1$ in a neighbourhood of infinity and has no other critical points.
 
{\em (ii)} Suppose $F$ is a $C^1$ mapping defined on the closed ball $\overline{B(p,r)}$ whose restriction to $\partial B(p,r)$ satisfies the assumptions of part {\em (i)}. Then $F$ extends to a $C^1$ branched cover $\hat F$ of the Riemann sphere.
\end{lemma}

\begin{proof}[Sketch of proof]
We may choose $\hat F$ so that it maps the ray $[p + re^{i\theta}, \infty)$ diffeomorphically onto the ray $[F(p + re^{i\theta}), \infty)$. In part (ii), a little care needs to be taken so that the extension $\hat F$ is $C^1$ on $\partial B(p,r)$.
\end{proof}

The assumptions of the above lemma hold when $r > 0$ is sufficiently small. 
We point out some minor differences in the computation:
\begin{itemize}
\item $\hat F$ is a rational map of degree $d = a + 1$. 
\item The degree of $\hat F_n$ drops by $D = d_n - d = d_n - (a + 1)$.
\item $B(p,r)$ contains $2D$ extra critical points of $F_n$, which disappear in the limit.
\item $a+1$ connected components of $\hat F_n^{-1}(\mathscr B)$ are outside $B(p,r)$.
\item $d_n - (a+1) = D$ connected components of $\hat F_n^{-1}(\mathscr B)$ are compactly contained in $B(p,r)$, so Lemma \ref{sum-degrees} remains true in this context.
\end{itemize}

\bigskip

\noindent Q. What happens if the limit $F$ is a constant function?

\medskip

\noindent A. Since bubbling is a local matter, we may assume that the domain $\Omega$ is bounded. We can then work with the perturbed sequence $F_{n, 1}(z) = F_n(z) + z$. Since
$$
 \frac{2|F'_n(z) + 1|}{1 + |F_n(z) + z|^2} 
  \, \asymp \,  \frac{2|F'_{n}(z)|}{1 + |F_{n}(z)|^2},
$$
when either quantity is large, the spherical areas $\mathscr A_{\sph}\bigl (F_n(z) +  z, \Omega \bigr )$ are uniformly bounded if and only if the $\mathscr A_{\sph}(F_n, \Omega)$ are. The construction in Section \ref{sec:tree} produces a tree of spheres $(\Sigma_k, \mathcal R_{k,1})$ for the perturbed sequence of meromorphic functions. The rational functions for the original sequence are simply $\mathcal R_k(z) = \mathcal R_{k, 1}(z) - p$, where $p$ is the point in $\Omega$ to which the branch of the tree containing $\Sigma_k$ is attached. (When we zoom in near $p \in S$, the function $z$ looks like a constant.)

\medskip

{\em Further remark.} In the locally univalent case, Li and Shafrir \cite{LS} observed that all bubbles are simple (correspond to rational functions of degree 1) and are attached directly to $\Omega$. In this case, the limit function has to be constant. This can be seen from the construction of the tree of spheres as the formation of non-simple bubbles and higher-order bubbles involves critical points. W.~Chen \cite{chen} has constructed sequences of meromorphic functions without critical points which feature an arbitrary number of simple bubbles.

\section{Preliminaries for removability}

We now turn to the second half of the paper, where we show that NED sets are removable for spherical Dirichlet functions.

\subsection{Notation}

Let $E \subset \Omega$ be a compact set and $F \in \mathcal F(\Omega \setminus E)$ be a spherical Dirichlet function.
For a rectifiable curve $\gamma \subset \Omega$, we denote the spherical length of its image counted with multiplicity by
$$
L(\gamma) = \int_{\gamma \setminus E} \frac{2|F'(z)|}{1+|F(z)|^2} \, |dz|.
$$
We will write $\Gamma$ for the domain enclosed by $\gamma$.
Similarly, for an open set $U \subset \Omega$, we write 
$$
\mathscr E(U) =  \int_{U \setminus E} \frac{4 |F'(z)|^2}{(1+|F(z)|^2)^2} \cdot |dz|^2.
$$

\subsection{NED sets}
\label{sec:NED}

In \cite[Theorem 3.1]{VG},  Vodopyanov and Goldshtein showed that NED sets are removable for continuous $W^{1,2}_{\loc}$ functions: if $F \in W^{1,2}_{\loc}(\Omega \setminus E, \mathbb{R})$ extends continuously to $\Omega$, then $F \in W^{1,2}_{\loc}(\Omega, \mathbb{R})$. Another proof is given in \cite[Theorem 7.8]{ntalampekos}. It is easy to see that the Vodopyanov-Goldshtein characterization of NED sets is also valid when one considers Sobolev spaces with values in the complex numbers $\mathbb{C}$ or the Riemann sphere $\sph$.
 \begin{remark}
(i) Since we are dealing with continuous functions, there is no difficulty in defining Sobolev spaces with values in the Riemann sphere by passing to coordinate charts.

(ii) From the  Vodopyanov-Goldshtein characterization of NED sets, it follows that in order to show that a compact set $E \subset \Omega$ is removable for $\mathcal F(\Omega)$, it is enough to show that any spherical Dirichlet function $F \in \mathcal F(\Omega \setminus E)$ extends to a continuous function from $\Omega$ to $\sph$. Indeed, the membership
 $F \in \mathcal F(\Omega \setminus E) \cap W^{1,2}(\Omega, \sph)$ implies that $F$ is weakly holomorphic (in the sense of distributions) on $\Omega$. We may then use Weyl's lemma to conclude that $F$ is strongly holomorphic on $\Omega$.
 \end{remark}

One of our main tools will be a lemma due to Ntalampekos  which allows one to perturb curves off NED sets. Recall that a family of curves $\Gamma_0 \subset \Omega$ has modulus zero, if for any $\varepsilon > 0$, there exists a measurable function $\lambda: \Omega \to [0,\infty)$ such that
$$
\int_\Omega \lambda(z)^2 |dz|^2 < \varepsilon \quad \text{and} \quad \int_\gamma \lambda(z) |dz| \ge 1, \qquad \gamma \in \Gamma_0.
$$

\begin{lemma}[Ntalampekos perturbation lemma]
\label{ntalampekos}
 Let $\Omega \subset \mathbb{C}$ be a domain in the plane and $\gamma$ be a rectifiable path in $\Omega$ that lies outside an exceptional curve family $\Gamma_0 = \Gamma_0(F)$ of modulus zero.
There exists a path $\hat \gamma$ with the same endpoints as $\gamma$ that avoids the set $E$ (except possibly, at the endpoints) with $$d_{\haus}(\gamma, \hat{\gamma}) < \varepsilon, \qquad |L(\hat{\gamma}) - L(\gamma)| < \varepsilon.$$
   If $\gamma$ is a Jordan curve, then one may choose the perturbed path $\hat{\gamma} \subset \Omega \setminus E$ to be a smooth Jordan curve.
\end{lemma}

The above statement is a consequence of \cite[Theorem 7.1, $(I) \Leftrightarrow (VI)$]{ntalampekos} with
$$
\rho(z) = \frac{2|F'(z)|}{1+|F(z)|^2} \cdot \chi_{\Omega \setminus E}(z).
$$

\begin{remark} 
(i) Since the metric $\rho$ is continuous on $\Omega \setminus E$, once we perturb a Jordan curve $\gamma$ off $E$, it is easy to perturb $\gamma$ further to obtain a smooth curve. 

(ii) Lemma \ref{ntalampekos} implies that NED sets are totally disconnected, however, the NED property is much more restrictive. For instance, by \cite[Theorem 10]{AB}, NED sets are {\em metrically removable}\/: any two points $x, y \in \mathbb{C}$ can be joined by a curve which avoids the set $E$ (except possibly, at the endpoints) whose Euclidean length is arbitrarily close to $|x-y|$. One can also deduce the metric removability of NED sets from \cite[Theorem 7.1]{ntalampekos} with $\rho = \chi_{\mathbb{C}}$.
For more information on metrically removable sets, we refer the reader to the work of Kalmykov, Kovalev and Rajala \cite{KKR}.
\end{remark}

\subsection{Length-area estimates}

We will use the following classical length-area estimate:

\begin{lemma}
\label{peripheral-curves}
Suppose $A = A(p; s, r) = \{ z \in \mathbb{C}: r < |z| < s \}$ is a round annulus contained in $\Omega$ of modulus $m = (1/2\pi) \log(r/s)$. There exists a simple closed curve $\gamma \notin \Gamma_0$ that separates the two boundary components such that
$$
L(\gamma)^2 \le C_1(m) \cdot \mathscr E(A).
$$
Similarly, there exists a simple curve $\delta \notin \Gamma_0$ that connects the two boundary components with
$$
L(\delta)^2 \le C_2(m) \cdot  \mathscr E(A).
$$
The constants $C_1(m)$ and $C_2(m)$ are uniform as $m$ ranges over a compact subset of $(0,\infty)$. 
The curve $\gamma$ can be taken to be a circle $\partial B(p, \rho)$ with $s < \rho < r$ while $\delta$ can be taken to be a radial line segment 
$[se^{i\theta}, re^{i\theta}]$.
\end{lemma}

The above lemma allows us to surround NED sets by short multicurves:

\begin{lemma}
\label{good-cover}

Let $E \subset \Omega$ be a compact NED set.
For any $\varepsilon > 0$, one can find a finite collection of smooth Jordan curves $\gamma_1, \gamma_2, \dots, \gamma_n \subset \Omega$ such that:

{\em (i)} the domains $\Gamma_i$ have disjoint closures and cover $E$,

{\em (ii)} $\sum_i L(\gamma_i)^2 < \varepsilon$.

{\em (iii)} $\sum_i \mathscr E(\Gamma_i) < \varepsilon$.

\noindent
If $U$ is an open set containing $E$, we can choose the curves $\gamma_i$ to lie in $U$.
\end{lemma}

\begin{proof}
Since NED sets have measure zero, by shrinking $U$ if necessary, we may assume that $\mathscr E(U) < \varepsilon$. We then choose $\delta > 0$ so that $\dist(E, \partial U) \ge 4\delta$.

Let $\mathcal A$ be the collection of round annuli $A(p; 2\delta, 3\delta) \subset U$ with $p \in \delta \mathbb{Z}^2$. For each annulus $A \in \mathcal A$, we can use Lemma \ref{peripheral-curves} to select a simple curve $\eta_A \subset A$ that separates the two boundary components of $A$ with $L(\eta_A)^2 \le C_1 \cdot \mathscr E(A)$. Using Lemma \ref{ntalampekos} to perturb the curves $\eta_A$ if necessary, we may assume that they don't pass through $E$.

 Since a point is contained in a bounded number of annuli $A \in \mathcal A$,
$$
\sum_{A \in \mathcal A} L(\eta_A)^2 \, \le \, C_1 \sum_{A \in \mathcal A}  \mathscr E(A) \, < \,  C_2 \varepsilon.
$$
Let $\{\Gamma_i\}$ be the bounded complementary components of
$\mathbb{C} \setminus \bigcup_{A \in \mathcal A} \eta_A$ which contain a point of $E$. Since $E$ is a compact set, the collection
$\{\Gamma_i\}$ is finite. Let $n$ be its cardinality.

Since each $\gamma_i = \partial \Gamma_i$, $i = 1, 2, \dots, n$, is composed of arcs from a bounded number of $\eta_A$,
$$
\sum_{i=1}^n L(\gamma_i)^2 < C_3\varepsilon.
$$
By shrinking the domains $\Gamma_i$ slightly (and changing $L(\gamma_i)$ and $\mathscr E(\Gamma_i$) by an arbitrarily small amount), we can make the $\gamma_i$ smooth and  $\overline{\Gamma_i}$ disjoint.
For instance, one can form the Riemann maps $\varphi_i: \mathbb{D} \to \Gamma_i$ and replace $\gamma_i$ with $\gamma_{i,r} = \varphi_i(\partial \mathbb{D}_r)$ for some $r$ sufficiently close to 1. Indeed, as the curves $\{\gamma_i\}$ are rectifiable, the associated Riemann maps $\{\varphi_i\}$ have derivative in the Hardy space $H^1$ and the $L(\gamma_{i,r})$ converge to $L(\gamma_i)$.

 To complete the proof, we replace $\varepsilon$ by 
$\min \bigl (\varepsilon,\, \varepsilon/C_3 \bigr )$ to get (ii).
\end{proof}

\subsection{An isoperimetric estimate}

Let $\omega = \frac{4 dx \wedge dy}{(1+|x|^2 +|y|^2)^2}$ denote the volume form on the sphere $\sph$. For a point $p  \in \sph$, let 
\begin{itemize}
\item $U_p$ be the hemisphere centered at $p$,
\item  $p^*$ be its diametrically opposite point.
\end{itemize}
Since $\hat{\mathbb{C}} \setminus \{p^*\}$ is contractible, there is a 1-form $\beta_p$ on $\hat{\mathbb{C}} \setminus \{p^*\}$ so that $\omega = d\beta_p$. By subtracting a 1-form with constant coefficients from $\beta_p$, we may assume that $\beta_p$ vanishes at $p$, so that
 $$
 \|\beta_p(x)\| \le C_p \dist_{\sph}(p, x), \qquad x \in U_p.
 $$
In particular, if $\eta$ is a curve passing through $p$ and is contained in $U_p$, then
\begin{equation}
\label{eq:stokes-estimate}
\biggl |\int_{\eta} \beta_p \biggr | \le C_p \cdot \ell_{\sph}(\eta)^2,
\end{equation}
where $\ell_{\sph}(\eta)$ is the length of $\eta$ as measured in the spherical metric $\rho_{\sph} = \frac{2|dz|}{1+|z|^2}$. If we select the 1-forms
$\{\beta_p\}_{p \in \hat{\mathbb{C}}}$ in an $\SO(3)$-invariant way, then the constant $C_p$ in (\ref{eq:stokes-estimate}) will be independent of $p$.

\begin{lemma}
 \label{basic-isoperimetry}
Suppose $\Gamma \setminus \bigcup_{i=1}^n \overline{\Gamma_i} \subset \mathbb{C}$ is a domain with smooth Jordan boundary and $F$ is a meromorphic function defined on a neighbourhood of $\Gamma \setminus \bigcup_{i=1}^n \overline{\Gamma_i}$. Then,
$$
\dist \Bigl ( \mathscr E \bigl ( \Gamma \setminus \bigcup \overline{\Gamma_i}  \bigr ), \ 4\pi\mathbb{Z} \Bigr) \le C \Bigl (\sum L(\gamma_i)^2 + L(\gamma)^2 \Bigr ),
$$
where $C > 0$ is a universal constant.
\end{lemma}

In the proof below, we think of $F(\gamma)$ and $F(\gamma_i)$ as parametrized curves and $F \bigl ( \Gamma \setminus \bigcup \overline{\Gamma_i} \bigr )$ as a parametrized surface immersed in the sphere.

\begin{proof}
We may assume that $L(\gamma)$  and $L(\gamma_i)$, $i = 1, 2, \dots, n$, are less than $\pi/2$, otherwise there is nothing to prove. Pick arbitrary points $p$ and $p_i$ on $F(\gamma)$ and $F(\gamma_i)$ respectively. By construction, $F(\gamma)$ and $F(\gamma_i)$ are contained in hemispheres $U$ and $U_i$ centered at $p$ and $p_i$ respectively.

Perturbing the curves $\gamma$ and $\gamma_i$ if necessary, we may assume that they do not pass through the critical points of $F$, so that  $F(\gamma)$ and $F(\gamma_i)$ are smooth immersed curves.

Choose smooth immersed disks $D \subset U$ with $\partial D = F(\gamma)$ and $D_i \subset U_i$ with $\partial D_i = F(\gamma_i)$, $i =1, 2, \dots, n$. Then,
$$
S \, = \, F \biggl ( \Gamma \setminus \bigcup_{i=1}^n \overline{\Gamma_i} \biggr ) \, \cup \, D \, \cup \, \bigcup_{i=1}^n D_i
$$
is a closed immersed surface in the sphere. Since the homology class of $S$ is an integral multiple of the homology class of the sphere,
$$
\int_{F(\Gamma \setminus \bigcup \overline{\Gamma_i})} \omega \, + \, \int_D \omega \, + \, \sum_{i=1}^n \int_{D_i} \omega \ = \ 4 \pi k,
$$
for some $k \in \mathbb{Z}$. By Stokes theorem,
 \begin{equation}
 \label{eq:basic-isoperimetry}
4 \pi k - \mathscr E \biggl ( \Gamma \setminus \bigcup_{i=1}^n \overline{\Gamma_i} \biggr ) \ = \,  \int_{F(\gamma)} \beta \, + \, \sum_{i=1}^n \int_{F(\gamma_i)} \beta_i.
 \end{equation}
It remains to bound the terms on the right hand side of (\ref{eq:basic-isoperimetry}) using (\ref{eq:stokes-estimate}).
\end{proof}

\section{An energy estimate}

In this section, we show the following estimate on the decay of energy:

\begin{lemma}
\label{decay-of-energy}
Let $F \in \mathcal F(\Omega \setminus E)$ be a spherical Dirichlet function.
If $B(p,r) \subset \Omega$ and $\mathscr E(B(p,r)) \le 2\pi$, then
\begin{equation}
\label{eq:decay-of-energy}
 \mathscr E(B(p,s)) \le (s/r)^\alpha \cdot \mathscr E(B(p,r)), \qquad 0 < s < r,
\end{equation}
where $\alpha > 0$ is a universal constant.
\end{lemma}

The proof of Lemma \ref{decay-of-energy} uses the following isoperimetric bound:

\begin{lemma}
\label{isoperimetry}
Suppose $\gamma \subset \Omega$ is a Jordan curve with $\mathscr E(\Gamma) < 2\pi$. If $\gamma \notin \Gamma_0$ is not exceptional, then
$
\mathscr E(\Gamma) \lesssim L(\gamma)^2.
$
\end{lemma}

\begin{proof}
By the Ntalampekos perturbation lemma, we may assume that $\gamma \subset \Omega \setminus E$ is a smooth Jordan curve. For any $\varepsilon > 0$, we can use Lemma \ref{good-cover} to produce a collection of smooth Jordan domains $\{\Gamma_i\}$ with disjoint closures such that
$$
E \cap \Gamma \subset \bigcup_{i=1}^n \Gamma_i \subset \Gamma, \qquad \sum_{i=1}^n \mathscr E(\Gamma_i) < \varepsilon, \qquad \sum_{i=1}^n L(\gamma_i)^2 < \varepsilon.
$$
By Lemma \ref{basic-isoperimetry}, we have
$$
 \mathscr E \biggl ( \Gamma \setminus \bigcup_{i=1}^n \overline{\Gamma_i} \biggr ) \, \lesssim \, L(\gamma)^2 + \sum_{i=1}^n L(\gamma_i)^2 \, \le \, 
 L(\gamma)^2 + \varepsilon,
$$
so that $\mathscr E(\Gamma) =  \mathscr E \bigl ( \Gamma \setminus \bigcup_{i=1}^n \overline{\Gamma_i} \bigr ) 
+  \sum_{i=1}^n \mathscr E(\Gamma_i)  \lesssim L(\gamma)^2 + 2\varepsilon$.
The lemma follows since $\varepsilon > 0$ was arbitrary.
\end{proof}

Following \cite{PW}, the proof of Lemma \ref{decay-of-energy} runs as follows:
  
 \begin{proof}[Proof of Lemma \ref{decay-of-energy}]
Suppose $0 < \rho < r$. By Lemma \ref{isoperimetry} and the Cauchy-Schwarz inequality,
\begin{align*}
 \mathscr E(B(p, \rho))  & \lesssim  \biggl ( \int_{\partial B(p,\rho) \setminus E} \frac{2|F'(z)|}{1+|F(z)|^2} \, |dz| \biggr )^2 \\
& \le 2\pi \rho \int_{\partial B(p,\rho) \setminus E} \frac{4 |F'(z)|^2}{(1+|F(z)|^2)^2}  \, |dz|  \\
& \lesssim \rho \cdot \frac{d\mathscr E(B(p,\rho))}{d\rho},
 \end{align*}
provided that $\partial B(p,\rho) \notin \Gamma_0$ is not exceptional.
 Rearranging, we get
 $$
 \frac{d}{d\rho} \log  \mathscr E(B(p, \rho)) \ge \frac{\alpha}{\rho},
 $$
 for some $\alpha > 0$. Since almost every circle $\partial B(p,\rho)$ is not exceptional, we can integrate with respect to $\rho$ to obtain (\ref{eq:decay-of-energy}).
 \end{proof}

\section{Continuity}

As explained in Section \ref{sec:NED}, in order to show that NED sets are removable for spherical Dirichlet functions (Theorem \ref{removability-thm}), it is enough to show:

\begin{lemma}
\label{continuity}
Suppose $\Omega \subset \mathbb{C}$ is a domain in the plane and $E \subset \Omega$ is a compact NED set. Any spherical Dirichlet function $F \in \mathcal F(\Omega \setminus E)$ is continuous as a map from $\Omega$ to the Riemann sphere.
\end{lemma}

 In fact, the argument below will show that $F$ is locally H\"older continuous with exponent $\alpha/2$, where $\alpha$ is the constant from Lemma \ref{decay-of-energy}.
 
 \begin{lemma}
 \label{systems-of-curves}
Given a point $p \in \Omega$ such that $B(p,r) \subset \Omega$ and $\mathscr E(B(p,2r)) \le 2\pi$, one can find two families of curves $\{\hat{\gamma}_n\}_{n=0}^\infty$ and $\{\hat{\ell}_n\}_{n=0}^\infty$ in $B(p, r) \setminus E$ such that:
\begin{enumerate}
\item $\hat \gamma_n$ is a Jordan curve contained in the annulus $A (p; 0.9 \, r/2^{n}, r/2^{n} )$, and separates its two boundary components.
\item $\hat \ell_n$ is a Jordan arc which connects $\hat \gamma_n$ and $\hat \gamma_{n+1}$, and except for its endpoints, lies in the topological annulus bounded by these two curves.
\item 
$
L(\hat \gamma_n) \lesssim  2^{- \alpha n /2} 
$
and
$
L(\hat \ell_n) \lesssim 2^{- \alpha n /2}.
$
\end{enumerate}
\end{lemma}

\begin{proof}
By Lemmas \ref{peripheral-curves} and \ref{decay-of-energy}, there exists a sequence of nested concentric circles $\gamma_n = \partial B(p, r_n)$, $n = 0, 1, 2, \dots$, with
 $0.9 \, r/2^{n} < r_n < r/2^{n}$ such that
$$
L(\gamma_n) \, \lesssim \, \sqrt{\mathscr E(B(p,r/2^{n-1}))} \, \lesssim \,  2^{- \alpha n /2},
$$
as well as a sequence of line segments 
$$\ell_n = \bigl [p+ (0.9 \, r/2^{n+1}) e^{i\theta_n},\, p+ (r/2^{n}) e^{i\theta_n} \bigr ]$$ with
$$
L(\ell_n) \, \lesssim \, \sqrt{\mathscr E(B(p,r/2^{n-1}))} \, \lesssim \,  2^{- \alpha n /2}.
$$
We first apply the Ntalampekos perturbation lemma to the circles $\gamma_n$ to produce the desired Jordan curves $\hat \gamma_n \in 
A  (p; 0.9 \, r/2^{n}, r/2^{n} ) \setminus E$. We then apply the
 Ntalampekos perturbation lemma to the line segments $\ell_n$ to obtain Jordan arcs $\tilde \ell_n \in \Omega \setminus E$ with the same endpoints, and obeying the same estimate. Since
 the arc $\tilde \ell_n$ runs from $\partial B (p, r/2^{n})$ to $\partial B  (p, 0.9 \, r/2^{n+1}  )$, a sub-arc $\hat{\ell}_n$ connects $\hat \gamma_n$ and $\hat \gamma_{n+1}$.
\end{proof}

\begin{proof}[Proof of Lemma \ref{continuity}]
Fix a ball $B(q,r/2) \subset \Omega$ with $\mathscr E \bigl (B(q,5 r/2) \bigr) \le 2\pi$. Suppose $p, p'$ are two points in $B(q,r/2) \setminus E$. In order to estimate the spherical distance between  $F(p)$ and $F(p')$, we connect $F(p), F(p')$ by a curve. To that end, we use  Lemma  \ref{systems-of-curves} to form four families of curves $\{\hat{\gamma}_n\}_{n=0}^\infty$,  $\{\hat{\ell}_n\}_{n=0}^\infty$,  $\{\hat{\gamma}_n'\}_{n=0}^\infty$ and $\{\hat{\ell}_n'\}_{n=0}^\infty$ associated to the pairs $(p, r)$ and $(p', r)$.

Since the Euclidean distance $s = |p-p'|$ between $p$ and $p'$ is less than $r$, there is a unique integer $m \ge 0$ so that
$$
\frac{r}{2^{m+1}} \le |p - p'| < \frac{r}{2^{m}}.
$$
It is easy to see the curves 
$$
\hat \gamma_m \subset A(p; 0.9\, r/2^{m}, r/2^{m}) \qquad \text{and} \qquad
 \hat \gamma'_m \subset A(p'; 0.9\, r/2^{m}, r/2^{m})$$ intersect: they both enclose the midpoint of the line segment $[p, p']$ but neither curve encloses the other.
  Pick an arbitrary point $p'' \in \hat \gamma_m \cap \hat \gamma'_m$ in the intersection. Concatenating pieces of curves from these four families (with indices $\ge m$) produces a curve 
$\gamma_{p \leftrightarrow p'} \subset \Omega$ which joins  $p$ to $p'$ and passes through $p''$  with $L(\gamma_{p \leftrightarrow p'}) \lesssim 
(s/r)^{\alpha/2}$. The proof is complete.
\end{proof}

\section*{Acknowledgements}
The author wishes to thank Dimitrios Ntalampekos, Anand Patel, Oliver Roth and Mikhail Sodin for many interesting conversations. The author is very grateful to the anonymous referee for many helpful suggestions that helped improved the readability of this paper. This research was supported by the Israeli Science Foundation (grant no.~3134/21).

\bibliographystyle{amsplain}

\end{document}